\documentclass[a4paper,10pt]{article}

\usepackage[latin1]{inputenc}
\usepackage[T1]{fontenc} 
\usepackage{amsmath,amsthm}
\usepackage{indentfirst}
\usepackage[dvips]{graphicx}
\usepackage{amsfonts,amssymb}
\usepackage{geometry}
\usepackage{hyperref}
\usepackage{stmaryrd}
\usepackage{fancyhdr}

\newtheorem{prop}{Proposition}[section]
\newtheorem{LM}{Lemma}[section]
\newtheorem{thm}{Theorem}[section]
\newtheorem{df}{Definition}[section]
\newtheorem{df-prop}{Definition-Proposition}[section]
\newtheorem{cor}{Corollary}[section]

\newcommand{\R}{\mathbb{R}}

\newcommand{\D}{\Delta}
\newcommand{\C}{\mathbb{C}}

\newcommand{\1}{\mathbf{1}}
\renewcommand{\d}{\delta}
\renewcommand{\o}{\mathfrak{O}}
\newcommand{\p}{\mathfrak{P}}

\newtheorem*{thm*}{Theorem}
\newtheorem*{thm**}{Theorem \ref{eqasai}}
\newtheorem*{thme**}{Theorem \ref{asaiprinc}}

\newtheoremstyle{pourlesremarques}{\topsep}{\topsep}{\normalfont}{}{\bfseries}{.}{ }{}
\theoremstyle{pourlesremarques}
\newtheorem{rem}{Remark}[section]
\def\adots{\mathinner{\mkern2mu\raise 1pt\hbox{.}\mkern 3mu\raise
4pt\hbox{.}\mkern1mu\raise 7pt\hbox{{.}}}}

\title {Unitary representations of GL(n,K) distinguished by a Galois involution, for K a p-adic field}
\author{Nadir Matringe}
\date{}

\begin{document}

 \maketitle

\begin{abstract}
Let $F$ be a $p$-adic field, and $K$ a quadratic extension of $F$. Using Tadic's 
classification of the unitary dual of $GL(n,K)$, 
we give the list of irreducible unitary representations of this group distinguished by $GL(n,F)$, in terms of distinguished discrete series. 
As it is known that a 
generalised Steinberg representation $St(\rho,k)$ is distinguished if and only if the cuspidal representation $\rho$ is $\eta^{k-1}$-distinguished, 
for $\eta$ the 
character of $F^*$ with kernel the norms of $K^*$, this actually gives a classification of distinguished 
unitary representations in terms of distinguished cuspidal representations.
\end{abstract}

\section*{Introduction}

 In the present work, we study, for $F$ a $p$-adic field, and $K$ a quadratic extension of $F$, 
smooth and complex unitary (which will be synonymous with unitarisable for us) representations of $GL(n,K)$ which admit on their space a 
nonzero invariant linear form under $GL(n,F)$. These unitary representations are called $GL(n,F)$-distinguished (or simply distinguished), and are conjectured to be the unitary part of the image of a functorial lift, in the Langlands' program, from $U(n,K/F)$ to $GL(n,K)$. 
\par Distinguished generic representations of $GL(n,K)$ have been classified in \cite{M11}, in terms of distinguished quasi-discrete series, using Zelevinsky's classification 
of generic representations. Here we do the same for distinguished irreducible unitary representations, using Tadic's classification 
of irreducible unitary representations. Our main result (Theorem \ref{unitdist}) is similar to the main result of \cite{M11}. 
However, to extend the result from generic unitary to 
irreducible unitary representations, we use different techniques. Our main tools are the Bernstein-Zelevinsky derivative functors, we apply ideas from 
\cite{B84}. For instance, the building blocks for unitary representations (the so called Speh representations) are not parabolically induced, 
hence one needs new methods to deal with these representations. That is what we do in Section \ref{sectionspeh}, to obtain a definitive statement in 
Corollary \ref{spehdist}, which we state here as a theorem.
\begin{thm*}
Let $k$ and $m$ be two positive integers, let $n$ be equal to $km$. If $\D$ is a discrete series of $GL(m,K)$, and $u(\D,k)$ is the corresponding 
Speh representation of $GL(n,K)$ (see Definition \ref{Speh}), then $u(\D,k)$ is distinguished if and only if $\D$ is.  
\end{thm*}
 One direction is given by 
the fact that if $\pi$ is a distinguished irreducible unitary representation, then it is also the case of its highest shifted 
derivative (see Proposition \ref{higherdistunitary}). The other direction is a nontrivial generalisation of the following simple observation: if $\rho$ 
is a distinguished cuspidal representation of 
$GL(n,K)$, then it is known that the parabolically induced representation $\nu^{1/2}\rho \times \nu^{-1/2}\rho$ is distinguished, and it is also known 
that its irreducible submodule $St(\rho,2)$ is not distinguished, hence its quotient $u(\rho,2)$, which is a Speh representation of $GL(2n,K)$, is 
distinguished. The case of general irreducible unitary representations of $GL(n,K)$ distinguished by $GL(n,F)$ is treated in 
Section \ref{sectiongeneral}. We obtain the main result of the paper in Theorem \ref{unitdist}. Denoting by $\sigma$ the nontrivial element 
of the Galois group of $K$ over $F$, by $\pi^\vee$ the smooth contragredient of a representation $\pi$ of $GL(n,K)$, and by $\pi^\sigma$ 
the representation $\pi\circ \sigma$, its statement is as follows.
\begin{thm*}
Let $n$ be a positive integer, and $\pi$ an irreducible unitary representation of $GL(n,K)$. By Tadic's classification 
(see Theorem \ref{tadicclassif}), the representation 
$\pi$ is a commutative product (in the sense of normalised parabolic induction) of representations of the form $u(\D,k)$ for 
$k>0$ and $\D$ a discrete series, and representations of the form $\pi(u(\D,k),\alpha)$ (see Definition \ref{dfrigid}), 
for $\D$ and $k$ as before, and $\alpha$ an element of $(0,1/2)$. Then the representation 
$\pi$ is distinguished if and only if $\pi^\vee$ is isomorphic to $\pi^\sigma$, and the Speh representations $u(\D,k)$ 
occuring in the product $\pi$ with odd multiplicity are distinguished.
\end{thm*}

\subsection*{Acknowledgments} I thank U.K. Anandavardhanan and A. Minguez for suggesting to study distinction for Speh representations, as well as 
helpful conversations. I also thank I. Badulescu for answering some questions about these representations. I thank S. Sugiyama for pointing out many 
typos. Finally, thanks to 
the referee's accurate reading and helpful comments, the general presentation of the paper improved significantly, 
and some proofs were clarified.

\section{Preliminaries}

\subsection{Basic facts and notations}\label{sectionsmooth}

First, in the following, we fix a non-archimedean local field $F$ of characteristic $0$, and an algebraic closure $\overline{F}$ of $F$.  
We denote by $K$ a quadratic extension of $F$ in $\overline{F}$. We denote by $\o_F$, $\p_F$, $\o_K$, $\p_K$, the respective ring of integers of 
$F$ and $K$ and their maximal ideals. We denote by $|.|_F$ and $|.|_K$ the normalised absolute values, they satisfy $|x|_K=|x|_F^2$ for $x$ in $F$.  
We fix a nontrivial smooth character $\theta$ of $K$, which is trivial on $F$. 
We denote by $\sigma$ the non trivial element of the Galois group $Gal_F(K)$ of $K$ over $F$, and by $\eta$ the quadratic 
character of $F^*$, whose kernel is the set of norms of $K^*$. For $n$ and $m\geq 1$, we denote by $\mathcal{M}_{n,m}$ the space of matrices 
$\mathcal{M}(n,m,K)$, by $\mathcal{M}_n$ the algebra 
$\mathcal{M}_{n,n}$, and by $G_n$ the group of invertible elements in $\mathcal{M}_n$. We will denote by $G_0$ the trivial group. 
If $m$ belongs to $\mathcal{M}_n$, we denote by $m^\sigma$ the matrix obtained from $m$ by applying $\sigma$ to each entry. 
If $S$ is a subset of $\mathcal{M}_n$, we denote by $S^\sigma$ the subset of $S$ consisting of elements fixed by $\sigma$.  
For $m\in \mathcal{M}_n$, we denote by $|m|_K$ or $\nu_K(m)$ the real number $|det(m)|_K$, and we define similarly $|m|_F$ or $\nu_F(m)$ 
for $m$ in $\mathcal{M}_n^\sigma$.
\par When $G$ is a closed subgroup of $G_n$, we denote by 
$Alg(G)$ the category of smooth complex $G$-modules. If $(\pi,V)$ belongs to $Alg(G)$, $H$ is a closed subgroup of $G$,
 and $\chi$ is a character of $H$, we denote by $V(H,\chi)$ the subspace of $V$ generated by vectors of the form $\pi(h)v-\chi(h)v$ 
for $h$ in $H$ and $v$ in $V$. 
This space is stable under the action of the subgroup $N_G(\chi)$ of the normalizer $N_G(H)$ of $H$ in $G$, which fixes $\chi$. 
We denote by $\delta_H$ the positive character of $N_G(H)$ such that if $\mu$ is a right Haar measure on $H$, and $int$ is the action 
of $N_G(H)$ on smooth functions $f$ with compact support in $H$, given by $(int(n)f)(h)=f(n^{-1}hn)$, then 
$\mu \circ int(n)= \delta_H(n)\mu $ for $n$ in $N_G(H)$. 
The space $V(H,\chi)$ is $N_G(\chi)$-stable. Thus, if $L$ is a closed-subgroup of $N_G(\chi)$, and $\delta'$ is a (smooth) character of 
$L$ (which will be a normalising character dual to that of normalised induction later), the quotient 
$V_{H,\chi}=V/V(H,\chi)$ (that we simply denote by $V_H$ when 
$\chi$ is trivial) becomes a smooth $L$-module for the (normalised) action $l.(v + V(H,\chi))= \delta'(l)\pi(l)v + V(H,\chi)$ of $L$ on 
$V_{H,\chi}$. If $(\rho,W)$ belongs to $Alg(H)$, we define the objects 
$(ind_H^G(\rho), V_c=ind_H^G(W))$ and $(Ind_H^G(\rho), V=Ind_H^G(W))$ of $Alg(G)$ as follows. The space $V$ is the space 
$\mathcal{C}^\infty(H\backslash G,\rho)$ of smooth functions 
from $G$ to $W$, fixed under right translation by the elements of a compact open subgroup 
$U_f$ of $G$, and  satisfying $f(hg)=\rho(h)f(g)$ for all $h$ in $H$ and $g$ in $G$. The space $V_c$ 
is the subspace $\mathcal{C}_c^\infty(H\backslash G,\rho)$ of $V$, consisting of functions with support compact mod $H$, in both cases, the action of $G$ is by right translation on the functions. By definition, the real part 
$Re(\chi)$ of a character $\chi$ of $F^*$ is the real number $r$ such that $|\chi(t)|_{\C}=|t|^r$, where $|z|_{\C}=\sqrt{z\bar{z}}$ for $z$ 
in $\C$. 

\subsection{Irreducible representations of GL(n)}\label{sectionirr}

We will only consider smooth representations of $G_n$ and its closed subgroups. 
We denote by $A_n$ the maximal torus of diagonal matrices in $G_n$. It will sometimes be useful to parametrize $A_n$ with simple roots, i.e. 
to write an element $t=diag(t_1,\dots,t_n)$ of $A_n$, as $t=z_1\dots z_n$, where $z_n=t_nI_n$, and $z_i=diag((t_i/t_{i+1})I_i,I_{n-i})$ 
belongs to the center of $G_i$ embedded in $G_n$, which we denote $Z_i$. For $z_i=diag(t_i I_i,I_{n-i})$ in $Z_i$, we denote $t_i$ by $t(z_i)$. 
If $n\geq 1$, let $\bar{n}=(n_1,\dots,n_t)$ be a partition of $n$ of length $t$ (i.e. an ordered
 set of $t$ positive integers 
whose sum is $n$), we denote by $M_{\bar{n}}$ the Levi subgroup of $G_n$, of matrices 
$diag(g_1,\dots,g_t)$, with each $g_i$ in $G_{n_i}$, by $N_{\bar{n}}$ the unipotent subgroup 
of matrices $\begin{pmatrix} I_{n_1} & \star & \star \\ & \ddots & \star \\ & & I_{n_t} \end{pmatrix}$, and by $P_{\bar{n}}$ the standard 
parabolic subgroup 
$M_{\bar{n}}N_{\bar{n}}$ (where $M_{\bar{n}}$ normalises $N_{\bar{n}}$). Note that $M_{(1,\dots,1)}$ is equal to $A_n$, and 
we set $N_{(1,\dots,1)}=N_n$. 
For each $i$, let $\pi_i$ be a smooth representation of $G_{n_i}$, then the tensor product 
$\pi_1 \otimes \dots \otimes \pi_t$ is a representation of $M_{\bar{n}}$, which can be considered as a representation of $P_{\bar{n}}$ trivial on 
$N_{\bar{n}}$. We will use the product notation  
$$\pi_1\times \dots \times \pi_t= Ind_{P_{\bar{n}}}^{G_n}(\d_{P_{\bar{n}}}^{1/2}\pi_1 \otimes \dots \otimes \pi_t )$$ for 
the normalised parabolic induction. 
\par We say that an irreducible 
representation $(\rho,V)$ of $G_n$ is cuspidal, if the Jacquet module $V_{N_{\bar{n}}}$ is zero whenever
 $\bar{n}$ is a proper partition of $n$. 
 Suppose that $\bar{n}=(m,\dots,m)$ is a partition of $n$ of length $l$, and that $\rho$ is a cuspidal representation of $G_m$. Let $a\leq b$ be two 
integers, such that $b-a+1= l$, then 
Theorem 9.3. of \cite{Z80} implies that the $G_n$-module $\nu_K^{a}\rho \times \dots \times \nu_K^{b}\rho$ has a 
unique irreducible quotient which we denote by $\D(\rho,b,a)$. We call it a segment or a quasi-discrete series of $G_n$. If in addition, a 
quasi-discrete series is unitary (which amounts to say that its central character is unitary), we will call it a discrete series 
or a unitary segment. We will sometimes write $St(\rho,l)=\D(\rho,(l-1)/2,-(l-1)/2).$\\  

We end this paragraph with a word about induced representations of Langlands' type and their quotients.

\begin{df}
Let $\D_1,\dots,\D_t$ be segments of respectively $G_{n_1},\dots,G_{n_t}$, and suppose that the central characters satisfy the relation 
$Re(c_{\D_i})\geq Re(c_{\D_{i+1}})$. Let $n=n_1+\dots+n_t$, then 
the representation $\D_1\times \dots \times \D_t$ of $G_n$ is said to be induced of Langlands' type.  
\end{df}

The following result is well-known, and can be found in \cite{R82}.

\begin{prop}\label{lquotient}
 Let $\pi=\D_1\times \dots \times \D_t$ be an induced representation of Langlands' type as above, then $\pi$ has a unique irreducible quotient, 
which we denote by $L(\D_1,\dots,\D_t)$. If $\D'_1,\dots,\D'_s$ are other segments with $Re(c_{\D'_j})\geq Re(c_{\D_{j+1}'})$, such that 
$L(\D_1,\dots,\D_t)= L(\D'_1,\dots,\D'_s)$, then we have the equality of non ordered sets 
$\{\D_1,\dots,\D_t\}=\{\D_1',\dots,\D_s'\}$. 
\end{prop}

A particular class of Langlands' quotients, is the class of Speh representations, which are the building blocks 
of the unitary dual of $G_n$, in Tadic's classification.

\begin{df}\label{Speh}
Let $k$ and $m$ be two positive integers, and set $n=km$. If $\D$ is a segment of $G_m$, we denote by $u(\D,k)$ the representation 
$L(\nu_K^{(k-1)/2}\D,\dots ,\nu_K^{(1-k)/2}\D)$ of $G_n$.  
\end{df}

We now recall some basic facts about Bernstein-Zelevinsky derivatives.

\subsection{Derivatives}\label{sectionder}

We define a character of $N_n$ which we denote again by $\theta$, by the formula $\theta(m)=\theta(\sum_{i=1}^{n-1} m_{i,i+1})$. 
For $n\geq2$ we denote by $U_{n}$ the group of matrices of the form $\begin{pmatrix} I_{n-1} & v \\& 1 \end{pmatrix}$. 
For $n>k\geq 1$, the group $G_{k}$ embeds naturally in $G_{n}$, and is given by matrices of the form 
$diag(g,I_{n-k})$. We denote by $P_n$ the mirabolic subgroup $G_{n-1}U_n$ of $G_n$ for $n\geq 2$, and we set $P_1=\{1_{G_1}\}$.
 If one sees $P_{n-1}$ as a subgroup of $G_{n-1}$ 
itself embedded in $G_n$, then $P_{n-1}$ is the normaliser of $\theta_{|U_n}$ in $G_{n-1}$ (i.e. if $g\in G_{n-1}$, then 
$\theta (g^{-1}ug)$ for all $u\in U_n$ if and only if $g\in P_{n-1}$). We define the following functors:

\begin{itemize}

\item The functor $\Phi^{+}$ from $Alg(P_{k-1})$ to $Alg(P_{k})$ such that, for $\pi$ in $Alg(P_{k-1})$, one has
$\Phi^{+} \pi = ind_{P_{k-1}U_k}^{P_k}(\delta_{U_k}^{1/2}\pi \otimes \theta)$.

\item The functor $\hat{\Phi}^{+}$ from $Alg(P_{k-1})$ to $Alg(P_{k})$ such that, for $\pi$ in $Alg(P_{k-1})$, one has
$\hat{\Phi}^{+} \pi = Ind_{P_{k-1}U_k}^{P_k}(\delta_{U_k}^{1/2}\pi \otimes \theta)$.

 \item The functor $\Phi^{-}$ from $Alg(P_k)$ to $Alg(P_{k-1})$ such that, if $(\pi,V)$ is a smooth $P_k$-module, 
$\Phi^{-} V =V_{U_k,\theta}$, and $P_{k-1}$ acts on $\Phi^{-}(V)$ by
 $\Phi^{-} \pi (p)(v+V(U_k,\theta))= \delta_{U_k} (p)^{-1/2}\pi (p)(v+V(U_k,\theta))$.

\item The functor $\Psi^{-}$ from $Alg(P_k)$ to $Alg(G_{k-1})$, such that if $(\pi,V)$ is a smooth $P_k$-module, 
$\Psi^{-} V =V_{U_k,1}$, and $G_{k-1}$ acts on $\Psi^{-}(V)$ by
 $\Psi^{-} \pi (g)(v+V(U_k,1))= \delta_{U_k} (g)^{-1/2}\pi (g)(v+V(U_k,1))$.

\item The functor $\Psi^{+}$ from $Alg(G_{k-1})$ to $Alg(P_{k})$, such that for $\pi$ in $Alg(G_{k-1})$, one has
$\Psi^{+} \pi = ind_{G_{k-1}U_k}^{P_k}(\delta_{U_k}^{1/2}\pi \otimes 1)=\delta_{U_k}^{1/2}\pi \otimes 1 $.

\end{itemize}
 
If $\tau$ is a representation of $P_n$ (or a representation of $G_n$, which we consider as a $P_n$-module by restriction), 
the representation $\tau^{(k)}$ of $G_{n-k}$ will be defined as $\Psi^{-}(\Phi^{-})^{k-1}\tau$, and will be called the $k$-th derivative of $\tau$.
It is shown in Section 3.5 of \cite{BZ77}, that these representations give a natural filtration of any $P_n$-module.

\begin{LM}\label{filtr}
If $\tau$ is an object of $Alg(P_n)$, then $\tau$ has a natural filtration of $P_n$-modules 
$0\subset \tau_{n} \subset \dots \subset \tau_{1}=\tau$, where $\tau_k= {\Phi^{+}}^{k-1} {\Phi^{-}}^{k-1}\tau$. Moreover the quotient 
$\tau_k/\tau_{k+1}$ is isomorphic to $(\Phi^{+})^{k-1}\Psi^+ \tau^{(k)} $ as a $P_n$-module. 
\end{LM} 

It is shown in Section 8 of \cite{Z80}, that if $\pi$ is an irreducible representation of $G_n$, then its highest derivative $\pi^-$, 
which is the derivative $\pi^{(k)}$, for $k\leq n$ maximal for the condition $\pi^{(k)}\neq 0$, is an irreducible representation of 
$G_{n-k}$. The following lemma, is an immediate consequence of Lemma 4.5. of \cite{BZ77}.

\begin{LM}\label{highderproduct}
Let $\pi_i$ be an irreducible representation of $G_{n_i}$, for positive integers $n_1,\dots,n_t$. Then the highest derivative of 
$\pi_1\times \dots \times \pi_t$ is the representation $\pi_1^-\times \dots \times \pi_t^-$.
\end{LM}

As we study unitary representations, we will need some further properties of these derivatives, which are extracted from \cite{B84}. 
First, as in this reference, we introduce the following definition.

\begin{df}
 Let $\tau$ be a $P_n$-module, we denote by $\tau^{[k]}$ the representation $\nu_K^{1/2}\tau^{(k)}$ of $G_{n-k}$, 
and call it the $k$-th shifted derivative of $\tau$. We denote by $\tau^{[-]}$ the highest shifted derivative of $\tau$.
\end{df}

We then recall the following consequence of the unitarisability criterion given in Section 7.3 of \cite{B84}. 

\begin{prop}\label{criterionbernstein}
If $\pi$ is an irreducible unitary representation of $G_n$, with highest derivative $\pi^{(h)}$. Then $\pi^{[h]}$ is 
unitary, and the central characters of the irreducible subquotients of $\pi^{[k]}$ all have positive real parts for $0<k<h$.
\end{prop}

\subsection{Unitary representations of GL(n)}\label{sectionunitary}

We now recall results from \cite{T86}, about the classification of irreducible unitary representations of $G_n$. 

\begin{df}\label{dfrigid}
For $\alpha \in \R$, $m>0$, $k>0$, and $\D$ a segment of $G_m$ we denote by $\pi(u(\D,k),\alpha)$ the representation 
$\nu_K^\alpha u(\D,k)\times \nu_K^{-\alpha} u(\D,k)$ of $G_n$, for $n=2mk$.
\end{df}

We now state Theorem D of \cite{T86}.

\begin{thm}\label{tadicclassif}
 Let $\pi$ be an irreducible unitary representation of $G_n$, then there is a partition $(n_1,\dots,n_t)$ of $n$ and 
representations $\pi_i$ of $G_{n_i}$, 
each of which is either of the form $\pi(u(\D_i,k_i),\alpha_i)$ for $\D_i$ a unitary segment, $k_i\geq 1$, and $0<\alpha_i<1/2$, or of the form 
$u(\D_i,k_i)$ for $\D_i$ a unitary segment and $k_i\geq 1$, such that $\pi=\pi_1\times \dots \times \pi_t$.
 Moreover, the representation $\pi$ is equal to $\pi'_1\times \dots \times\pi'_s$, for representations $\pi'_j$ of the same type as the representations $\pi_i$, 
if and only if $\{\pi_1,\dots,\pi_t\}=\{\pi'_1,\dots,\pi'_s\}$ as non ordered sets.
\end{thm}

If all the representations $\pi_i$ in the above theorem, are such that $k_i=1$, we say that $\pi$ is a \textit{generic} 
unitary representation of $G_n$.\\

We will also need the description of the composition series of the so-called end of complementary series, which is proved in 
\cite{T87} (see Theorem 2 of \cite{B11} for a quick proof). If $\D$ is the segment 
$St(\rho,l)$ for $l\geq 1$, we write $\D_+= St(\rho,l+1)$, and $\D_-= St(\rho,l-1)$, where $St(\rho,0)$ is $\1_{G_0}$ by convention.

\begin{thm}\label{end}
Let $m$ be a positive integer, $\D$ a segment of $G_m$, $k\geq 2$ an integer, and $n=2mk$. The representation 
$\pi(u(\D,k),1/2)$ of $G_n$ is of length $2$, and its irreducible subquotients are 
$u(\D_-,k)\times u(\D_+,k)$ and $u(\D,k-1)\times u(\D,k+1)$.
\end{thm}

Finally, we recall the formula which gives the highest shifted derivative of a Speh representation, from Section 6.1 of \cite{T86-2} 
(see (3.3) of \cite{OS08} for the proof). 

\begin{prop}\label{highderspeh}
Let $m>0$ and $k>1$ be two integers, and let $\D$ be a segment of $G_m$. The highest shifted derivative of the representation $u(\D,k)$ 
is equal to $u(\D,k)^{[m]}=u(\D,k-1)$. The highest (shifted) derivative of $\D$ is equal to $\1_{G_0}$.
\end{prop}

\subsection{Distinguished representations of GL(n)}\label{sectiondist}

In this paragraph, we recall results from \cite{M11}. First, we introduce some notations and definitions. 

\begin{df}
Let $G$ be a closed subroup of $G_n$, $H$ a closed subgroup of $G$, and $\chi$ a character of $H$, we say that a representation $\pi$ in $Alg(G)$ is 
$(H,\chi)$-distinguished if the space $Hom_H(\pi,\chi)$ is nonzero. If $H$ is clear, we say $\chi$-distinguished instead of $(H,\chi)$-distinguished, 
and if $\chi$ is trivial, we say $H$-distinguished (or distinguished if $H$ is clear). If $G=G_n$, and $H=G_n^\sigma$, we will sometimes 
say $(\sigma,\chi)$-distinguished instead of $(H,\chi)$-distinguished, and if $\chi$ is trivial, we will simply say $\sigma$-distinguished. 
\end{df}

We recall the following general facts from \cite{F91}, about $\sigma$-distinguished representations of $G_n$. We denote by $\pi^\sigma$ the representation 
$g\mapsto \pi(g^\sigma)$ for $\pi$ a representation of $G_n$.

\begin{prop}\label{autodualmult1}
Let $n\geq 1$ be an integer, and $\pi$ be an irreducible representation of $G_n$. If $\pi$ is $\sigma$-distinguished, then $\pi^\vee=\pi^\sigma$, 
and $Hom_{G_n^\sigma}(\pi,\1)$ is of dimension $1$.
\end{prop}

We now introduce the class of $\sigma$-induced irreducible unitary representations of $G_n$. They will turn out to be the $\sigma$-distinguished 
irreducible 
unitary representations of $G_n$.

\begin{df}
For $n\geq 1$, let $\pi$ be an irreducible unitary representation 
$$\pi=u(\D_1,k_1)\times \dots \times u(\D_s,k_s)\times \pi(u(\D_{s+1},k_{s+1}),\alpha_{s+1})\times \dots \times \pi(u(\D_{t},k_{t}),\alpha_t)$$ of $G_n$, 
with unitary segments $\D_i$, positive integers $k_i$, and $\alpha_i\in (0,1/2)$. The representation $\pi$ is said to be $\sigma$-induced, if it satisfies $\pi^\vee=\pi^\sigma$, and if for 
every $i\leq s$, such that $u(\D_i,k_i)$ occurs with odd multiplicity in the product $\pi$, the segment $\D_i$ is $\sigma$-distinguished.
\end{df}

\begin{rem}\label{onlyrem}
Maybe the preceding definition is not completely transparent to the reader. Let us try to explain how $\sigma$-induced irreducible unitary representations 
look like. Let 
$$\pi=u(\D_1,k_1)\times \dots \times u(\D_t,k_s)\times 
\pi(u(\D_{s+1},k_{s+1}),\alpha_{k_{s+1}})\times \dots \times \pi(u(\D_{t},k_{t}),\alpha_{k_{t}})$$
be an irreducible unitary representation of $G_n$.  First, if one has $\pi^\vee=\pi^\sigma$ (call this relation $\sigma$-self-duality), 
then it means the two following things:
\begin{itemize}
\item[a)] for $i$ between $1$ and $s$, either $u(\D_i,k_i)$ is $\sigma$-self-dual, or if this relation is not satisfied, 
there is $j\neq i$ between $1$ and $s$, such that $u(\D_j,k_j)^\vee= u(\D_i,k_i)^\sigma$.
\item[b)] for $i$ between $s+1$ and $t$, either $\pi(u(\D_i,k_i),\alpha_i)$ is $\sigma$-self-dual, or if this relation is not satisfied, 
there is $j\neq i$ between $s+1$ and $t$, such that $\pi(u(\D_j,k_j),\alpha_j)^\vee= \pi(u(\D_i,k_i),\alpha_i)^\sigma$.
\end{itemize}
In a) above, if you have $u(\D_i,k_i)^\vee= u(\D_i,k_i)^\sigma$ which occurs with multiplicity $\geq 2$, i.e. if there is $j\neq i$ between $1$ 
and $s$ such that $u(\D_j,k_j)=u(\D_i,k_i)$, then one has $u(\D_j,k_j)^\vee=u(\D_i,k_i)^\sigma$. Hence a) can also be stated 
as:\\
a') $u(\D_1,k_1)\times \dots \times u(\D_s,k_s)$ is a product of representations of the form $u(\D_i,k_i)\times (u(\D_i,k_i)^\vee)^\sigma$, and 
of $\sigma$-self dual representations $u(\D_j,k_j)$ which occur with odd multiplicity.\\
Now in b), if $\pi(u(\D_i,k_i),\alpha_i)$ is $\sigma$-self dual, it is equal to 
$\nu_K^{\alpha_i} u(\D_i,k_i) \times ((\nu_K^{\alpha_i} u(\D_i,k_i))^\vee)^\sigma$ (because $\D_i^\vee$ must be equal to $\D_i^\sigma$). All in all, $\pi$ is $\sigma$-self dual if and only if it is a product of representations of the form 
$$\nu_K^{\alpha} u(\D,k) \times ((\nu_K^{\alpha} u(\D,k))^\vee)^\sigma$$ for $0\leq \alpha <1/2$, $\D$ a discrete series and $k$ a positive integer (we allow here $\alpha$ to be equal to 
zero, in order to take in account representations $u(\D_i,k_i)\times (u(\D_i,k_i)^\vee)^\sigma$ occuring in a')), of representations of the form
$$\pi(u(\D,k),\alpha)\times (\pi(u(\D,k),\alpha)^\vee)^\sigma$$ for $\alpha$ in $(0,1/2)$ and $\D$ and $k$ as above, and of representations of the form $u(\D',k')$ ($\D'$ unitary, and $k'>0$) occuring with odd multiplicity, and which are $\sigma$-self dual. 
In this situation, $\pi$ is $\sigma$-induced if and only if these representations $u(\D',k')$ are such that $\D'$ is $\sigma$-distinguished.
\end{rem}

Theorem 5.2. of \cite{M11} then classifies distinguished generic representations.

\begin{thm}\label{distgen}
For $n\geq 1$, a generic unitary representation of $G_n$ is $\sigma$-distinguished if and only if it is $\sigma$-induced.
\end{thm}

We also recall Corollary 3.1 of \cite{M09} about distinction of discrete series.

\begin{prop}\label{discrdist}
Let $\rho$ be a cuspidal representation of $G_r$ for $r\geq 1$, and $\D=St(\rho,l)$ for $l\geq 1$. The segment $\D$ of $G_{lr}$ is $\sigma$-distinguished 
if and only if $\rho$ is $(\sigma,\eta^{l-1})$-distinguished.
\end{prop}

Finally, Corollary 1.6 of \cite{AKT04} says that the segment $\D$ above cannot be $\sigma$-distinguished and 
$(\sigma,\eta)$-distinguished at the same time. This has the following immediate corollary.

\begin{cor}\label{1sur2}
Let $\D$ be a segment of $G_n$ for $n\geq 2$, then $\D$ is $\sigma$-distinguished if and only if $\D_+$ is $(\sigma,\eta)$-distinguished. In particular, 
if $\D$ is distinguished, $\D_+$ is not.
\end{cor}

\section{Distinguished unitary representations}

We will first prove the convergence of integrals defining invariant linear forms.

\subsection{Asymptotics in the degenerate Kirillov model}\label{sectionasympt}

We denote by $N_{n,h}$ the group of matrices $h(a,n)=\begin{pmatrix} a & x \\ 0 & n \end{pmatrix}$, with $a$ in $G_{n-h}$, $n$ in $N_h$, and 
$x$ in $\mathcal{M}_{n-h,h}$.
It is proved in Section 5 of \cite{Z80}, that any irreducible representation $\pi$ of $G_n$, has a "degenerate Kirillov model" (which is 
just the standard Kirillov model in the non-degenerate case). This means that the restriction of $\pi$ to $P_n$, embeds as 
a unique $P_n$-submodule $K(\pi,\theta)$ of $(\hat{\Phi}^+)^{h-1}\Psi^+(\pi^{(h)})$, where $\pi^{(h)}=\pi^{-}$. The space $K(\pi,\theta)$ 
consists of smooth functions $W$ from $P_n$ to $V_{\pi^{(h)}}$, which are fixed under right translation by an open subgroup $U_W$, and satisfy 
the relation $$W(h(a,n)p)= |a|_K^{h/2}\theta(n)\pi^{(h)}(a)W(p)$$ for $h(a,n)$ in $N_{n,h}$, and $p$ in $P_n$. It can be handy to identify such a function with a map from $P_n$ to $V_{\pi^{[h]}}$, which satisfies the relation 
\begin{equation}\label{relkirillov} W(h(a,n)p)= |a|_K^{(h-1)/2}\theta(n)\pi^{[h]}(a)W(p)\end{equation} for $h(a,n)$ in $N_{n,h}$, and $p$ in $P_n$.\\

We now give an asymptotic expansion of the elements of $K(\pi,\theta)$, in terms of the exponents of $\pi$. The proof, which is omitted, is an 
easy adaptation of the proof of Theorem 2.1. in \cite{M11-2}. We write $\mathcal{C}_c^\infty(F,V)$ for the space of smooth functions with compact support from $F$ to a 
complex vector space $V$.

\begin{thm}\label{asympt}
Let $\pi$ be an irreducible representation of $G_n$, for $n\geq 2$. Let $\pi^{(h)}$ be the highest derivative 
of $\pi$, and let $W$ belong to $K(\pi,\theta)$. We suppose that we have $h\geq 2$, and we 
denote by $(c_{k,i_k})_{1\leq k \leq r_k}$ the family of central characters of the irreducible subquotients of $\pi^{(k)}$. In this situation, the restriction $W(z_{n-h+1}\dots z_{n-1})$ 
of $W$ to the torus $Z_{n-h+1}\dots Z_{n-1}$ is a linear combination of functions of the form 
$$\prod_{k=n-h+1}^{n-1} [c_{i_k,k}\d_{U_{k+1}}^{1/2}\dots \d_{U_{n}}^{1/2}](z_k) v_F(z_k)^{m_k}\phi_k(t(z_k))$$
 for $i_k$ between $1$ and $r_k$, non negative integers $m_k$, and functions $\phi_k$ in $\mathcal{C}_c^\infty(F,V_{\pi^{(h)}})$. 
\end{thm}

From this, we deduce the convergence of the following integrals, which we will need later.

\begin{prop}\label{convint}
Let $\pi$ be an irreducible unitary representation of $G_n$, for $n\geq 1$. Let $\pi^{(h)}$ be the highest derivative of $\pi$, and let $W$ 
belong to $K(\pi,\theta)$. We suppose that there is a nonzero $G_{n-h}^\sigma$-invariant linear form $L$ on the space of 
$\pi^{[h]}$, and for every element $W$ of $K(\pi,\theta)$, we define the map $f_{L,W}=L\circ W$. Then for all $W$ 
in $K(\pi,\theta)$, the integral 
$$\Lambda(W)=\int_{N_{n,h}^\sigma\backslash P_n^\sigma} f_{L,W}(p)dp$$ is absolutely convergent, and $\Lambda$ defines a 
nonzero $P_n^{\sigma}$-invariant 
linear form on $V_{\pi}$.
\end{prop}
\begin{proof} If $h$ equals $1$, then $\Lambda(W)$ is equal to $L(W(I_n))$ up to normalisation, and the result is obvious.  
For $h\geq 2$, first, thanks to Relation (\ref{relkirillov}), the restriction of the map $f_{L,W}$ to $P_n^\sigma$ satisfies the relation 
$$f_{L,W}(h(a,n)p)= |a|_F^{h-1}f_{L,W}(p)$$ for $p$ in $P_n^\sigma$, and $h(a,n)$ in $N_{n,h}^\sigma$. We notice that 
$|a|_F^{h-1}$ is indeed equal to $\frac{\d_{N_{n,h}^\sigma}}{\d_{P_n^\sigma}}(h(a,n))=\frac{|a|_F^{h}}{|a|_F}$. 
Actually, the integral $\Lambda(W)$ is equal to $$\int_{N_{n-1,h}^\sigma\backslash G_{n-1}^\sigma} f_{L,W}(p)dp.$$
Hence, thanks to the Iwasawa decomposition, the integral $\Lambda$ will converge absolutely for any $W$ in $K(\pi,\theta)$, if and only if so does 
the integral $$\int_{Z_{n-h+1}\dots Z_{n-1}}f_{L,W}(z_{n-h+1}\dots z_n)\d_{N_{n-1,h}^\sigma}^{-1}(z_{n-h+1}\dots z_{n-1})d^*z_{n-h+1}\dots d^*z_{n-1}$$ for any $W$ in $K(\pi,\theta)$. 
As $\d_{N_{n-1,h}^\sigma}(z_{n-h+1}\dots z_{n-1})$ is equal to the product 
$$\prod_{k=n-h+1}^{n-1} \d_{U_{k+1}^{\sigma}}\dots \d_{U_{n-1}^{\sigma}}(z_k)=\prod_{k=n-h+1}^{n-1} \d_{U_{k+1}}^{1/2}\dots \d_{U_{n-1}}^{1/2}(z_k)$$ 
for the $z_i$'s in $Z_i^\sigma$, we obtain that the integral $$\int_{Z_{n-h+1}\dots Z_{n-1}}|f_{L,W}(z_{n-h+1}\dots z_n)|\d_{N_{n-1,h}^\sigma}^{-1}(z_{n-h+1}\dots z_{n-1})d^*z_{n-h+1}\dots d^*z_{n-1}$$ is majorized by a sum of integrals of the form 
$$\displaystyle \prod_{k=n-h+1}^{n-1} \int_{Z_k} c_{i_k,k}\d_{U_{n}}^{1/2}(z_k) v_F(z_k)^{m_k}f_k(t(z_k))d^*z_k$$ 
$$= \prod_{k=n-h+1}^{n-1} \int_{Z_k} c_{i_k,k}\d_{U_{k+1}}^{1/2}(z_k) v_F(z_k)^{m_k}f_k(t(z_k))d^*z_k,$$ for functions 
$f_k=L\circ \phi_k$ in $\mathcal{C}_c^\infty(F)$, thanks to Theorem \ref{asympt}. These last integrals are convergent, as, according to Proposition 
\ref{criterionbernstein}, the real part $Re(c_{i_k,k}\d_{U_{k+1}}^{1/2})$ is positive. This concludes the proof of the convergence. To show that $\Lambda$ is nonzero, we just need to remember that 
$\pi$ contains as a $P_n$-submodule the space $(\Phi^+)^{h-1}(\Psi^+(\pi^{(h)}))$, and the restriction to $P_n^\sigma$ of elements of 
$(\Phi^+)^{h-1}(\Psi^+(\pi^{(h)}))$ is surjective on the space 
$$\mathcal{C}_c^\infty(N_{n,h}^\sigma \backslash P_n^\sigma,\frac{\d_{N_{n,h}^\sigma}}{\d_{P_n^\sigma}}\pi^{[k]}\otimes \1).$$
\end{proof}

\subsection{The case of Speh representations}\label{sectionspeh}

The aim of this section is to prove that a representation $u(\D,k)$ is $\sigma$-distinguished if and only if $\D$ is, independently of $k$. Oddly enough, 
the trickiest part is to prove that when $\D$ is $\sigma$-distinguished, so is $u(\D,k)$. We first recall as a lemma Proposition 1 of \cite{K04}, which is the key ingredient of the proof of the functional equation of the local Asai $L$-function.

\begin{LM}\label{lmkable}
Let $\tau$ be a representation of $P_n$ for $n\geq 1$, then the space $Hom_{P_{n+1}^\sigma}(\Phi^+(\tau),\1)$ is isomorphic 
to $Hom_{P_n^\sigma}(\tau,\1)$.
\end{LM}

This implies the following generalisation of Theorem 1.1 of \cite{AKT04}:

\begin{prop}\label{higherdistunitary}
Let $\pi$ be an irreducible unitary representation of $G_n$ for $n\geq 1$. The representation $\pi$ is $P_n^\sigma$-distinguished if 
and only if its highest shifted 
derivative $\pi^{[-]}$ is $\sigma$-distinguished.
\end{prop}
\begin{proof}
One implication follows from Proposition \ref{convint}. For the other one, we first notice that by definition of $\Psi^+$, if $\pi'$ is a 
representation of $G_k$ for $k\geq 0$, then 
the space $Hom_{P_{k+1}^\sigma}(\Psi^+(\pi'),\1)$ is isomorphic to $Hom_{G_k^\sigma}(\nu^{1/2}\pi',\1)$. Hence, thanks to Lemma \ref{lmkable}, 
the space $Hom_{P_{k+l}^\sigma}((\Phi^+)^{l-1}\Psi^+(\tau),\1)$ is isomorphic to $Hom_{G_k^\sigma}(\nu^{1/2}\tau,\1)$. Now, if $\pi$ is an 
irreducible unitary 
representation of $G_n$, let $h$ be the integer such that $\pi^{-}=\pi^{(h)}$. The restriction of $\pi$ to $P_n$ has a filtration with 
factors $(\Phi^+)^{k-1}\Psi^+(\pi^{(k)})$ for $k$ between $1$ and $h$ according to Lemma \ref{filtr}. If $L$ is a nonzero $P_n^\sigma$-invariant 
linear form on $\pi$, 
it must induce a nonzero element of $Hom_{P_n^\sigma}((\Phi^+)^{k-1}\Psi^+ \pi^{(k)},\1)\simeq Hom_{G_{n-k}^\sigma}(\pi^{[k]},\1)$ 
for some $k$ in $\{1,\dots,h\}$. But 
if the space $Hom_{G_{n-k}^\sigma}(\pi^{[k]},\1)$ is nonzero, it implies that the central character of one of the irreducible
 subquotients of $\pi^{[k]}$ has real part equal to zero, because $F^*$ must act trivially on at least one irreducible subquotient of $\pi^{[k]}$. Hence, according to Proposition \ref{criterionbernstein}, this means that the space $Hom_{P_n^\sigma}((\Phi^+)^{k-1}\Psi^+ \pi^{(k)},\1)$ is 
reduced to zero for $k$ between $1$ and $h-1$, and that the space 
$$Hom_{P_n^\sigma}((\Phi^+)^{h-1}\Psi^+\pi^{(h)},\1)\simeq Hom_{G_{n-h}^\sigma}(\pi^{[h]},\1)$$ is nonzero. The result is thus proved.
\end{proof}

The proof of the preceding proposition implicitly contains the following statement.

\begin{prop}\label{crucial}
Let $\pi$ be an irreducible unitary representation of $G_n$ which is $P_n^\sigma$-distinguished. Then its highest shifted derivative 
$\pi^{[-]}$ is $\sigma$-distinguished, and the 
space $Hom_{P_n^\sigma}(\pi,\1)$ is of dimension $1$, with basis a certain linear form $L$. Moreover, the restriction of $L$ to 
$\tau_{0}= (\Phi^+)^{h-1}\Psi^+(\pi^{-})$ is nonzero, and if $\tau$ is any $P_n$-submodule of $\pi$ which is $P_n^{\sigma}$-distinguished, 
then $\tau$ contains $\tau_{0}$, and the space $Hom_{P_n^\sigma}(\tau,\1)$ is spanned by the restriction $L_{|\tau}$.
\end{prop}

From this, we deduce a statement which will be used in a crucial way twice after.

\begin{prop}\label{crux}
Let $n_1$ and $n_2$ be two positive integers, and $\pi_1$ and $\pi_2$ be two irreducible unitary representations of $G_{n_1}$ and $G_{n_2}$ respectively. Suppose that 
$\pi_1$ is $G_{n_1}^\sigma$-distinguished, and that $\pi_2$ is $P_{n_2}^\sigma$-distinguished. In this situation, if $\pi=\pi_1\times \pi_2$ is 
$G_n^\sigma$-distinguished, then $\pi_2$ is $G_{n_2}^\sigma$-distinguished.
\end{prop}
\begin{proof}
We write $\pi_1\times \pi_2$ as induced from the lower parabolic subgroup $P^-=P_{(n_1,n_2)}^-$ obtained by transposing $P_{(n_1,n_2)}$. It is thus the 
space $\mathcal{C}_c^\infty(P^-\backslash G_n,\d_{P^-}^{1/2}\pi_1\otimes\pi_2)$. The double class $P^-P_n$ being open in $G_n$, this space contains 
$\tau=\mathcal{C}_c^\infty(P^-\backslash P^-P_n,\d_{P^-}^{1/2}\pi_1\otimes\pi_2)$, which is a $P_n$-submodule of $\pi$. Let $L_1$ be a basis of 
$Hom_{G_{n_1}^\sigma}(\pi_1,\1)$, $L_2$ be a basis of $Hom_{P_{n_2}^\sigma}(\pi_2,\1)$, and denote by $\lambda$ the linear form 
$L_1\otimes L_2$ on $\pi_1\otimes \pi_2$. We now introduce the following linear form on 
$\tau$: $$L: f\mapsto \int_{P^-\cap P_n^\sigma \backslash P_n^\sigma} \lambda(f(p)) dp.$$
It is well defined because the restriction of $f$ to $P_n^\sigma$ has compact support modulo $P^-\cap P_n^\sigma$, because it satisfies 
$f(hp)=|a|_F^{-n_2}|b|_F^{n_1}f(p)$ for 
$$h=\begin{pmatrix}a & 0 & 0 \\ x & b &y \\ 0 & 0 & 1 \end{pmatrix}\in P^-\cap P_n^\sigma$$ written in blocks according to the partition 
$(n_1,n_2-1,1)$ of $n$, and because of the relation 
$$\frac{\d_{P^-\cap P_n^\sigma}}{\d_{P_n^\sigma}}(h)= \frac{|a|_F^{1-n_2}|b|_F^{1+n_1}}{|a|_F|b|_F}=|a|_F^{-n_2}|b|_F^{n_1}.$$
Let's now show that $L$ is nonzero. For $v_1$ and $v_2$ in $V_{\pi_1}$ de $V_{\pi_2}$, let $U$ be a congruence subgroup of $G_n$, such that 
$U\cap G_{n_1}$ fixes $v_1$ and $U\cap G_{n_2}$ fixes $v_2$. As $U$ has an Iwahori decomposition with respect to $P^-$, the map 
defined by $f_{U,v_1,v_2}(p^- u)=\d_{P^-}^{1/2}\pi_1\otimes \pi_2(p^-)(v_1\otimes v_2)$ for $u$ in $U$, $p^-$ in $P^-$, and by zero outside $P^-U$, 
belongs to $V_\pi$. Moreover $L(f_{U,v_1,v_2})$ is a positive multiple of $L_1(v_1)L_2(v_2)$, in particular $L$ is nonzero. 
This implies that $L$ belongs to 
$Hom_{P_n^\sigma}(\tau,\1)-\{0\}$. It remains to prove that $\pi_2$ is $G_{n_2}^\sigma$-distinguished, we are going to prove that 
$L_2$ is actually $G_{n_2}^\sigma$-invariant. By Proposition \ref{crucial}, as $\pi$ is irreducible, unitary, and $\sigma$-distinguished, we know that 
$Hom_{P_n^\sigma}(\pi,\1)$ is one-dimensional, spanned by a linear form $L'$. Moreover, by the same proposition, up to multiplying $L'$ by a scalar, the restriction of $L'$ to 
$\tau$ is equal to $L$. Hence we denote $L'$ by $L$. The fact that $Hom_{P_n^\sigma}(\pi,\1)$ is one-dimensional also implies that $L$ is in fact 
$G_n^\sigma$-invariant. Now take $h$ of the form $diag(I_{n_1},b)$, with $b$ in $G_{n_2}(\o_K)$, we have $\rho(h)f_{U,v_1,v_2}=f_{U,v_1,\rho(b)v_2}$. 
If moreover $b$ belongs to $G_{n_2}(\o_K)^\sigma$, the relation $L(\rho(h)f_{U,v_1,v_2})=L(f_{U,v_1,v_2})$ implies the equality 
$L_1(v_1)L_2(\rho(b)v_2)=L_1(v_1)L_2(v_2)$. This implies that $L_2$ is $G_{n_2}(\o_K)^\sigma$-invariant, in particular it is $w_{n_2}$-invariant, for 
$w_{n_2}$ the antidiagonal matrix with ones on the second diagonal. As $L_2$ is $P_{n_2}^\sigma$-invariant by hypothesis, it is 
$G_{n_2}^\sigma$-invariant because 
$w_{n_2}$ and $P_{n_2}^\sigma$ span the group $G_{n_2}^\sigma$, and this concludes the proof.
\end{proof}

For Speh representations, we first obtain the following criterion of $P_n^\sigma$-distinction.

\begin{prop}\label{GtoPspeh}
Let $r$ be a positive integer, $k$ be an integer $\geq 2$, and $n=kr$. Let $\D$ be a discrete series of $G_r$, 
then the representation $u(\D,k)$ is $P_n^\sigma$-distinguished if and only if $u(\D,k-1)$ is $\sigma$-distinguished.
\end{prop}
\begin{proof}
We recall from Proposition \ref{highderspeh}, that $u(\D,k)^{[-]}$ is equal to $u(\D,k-1)$. We then apply Proposition \ref{higherdistunitary}.
\end{proof}

Proposition \ref{higherdistunitary} also has the following corollary.

\begin{cor}\label{uniformproddist}
Let $n_1,\dots, n_t$ and $k$ be positive integers, and $\D_i$ be a unitary segment of $G_{n_i}$ for each $i$. If the product $u(\D_1,k)\times \dots \times u(\D_t,k)$ is $\sigma$-distinguished, then the product $\D_1\times \dots \times \D_t$ is $\sigma$-distinguished as well.
\end{cor}
\begin{proof}
First, according to Theorem \ref{tadicclassif}, the product $u(\D_1,k)\times \dots \times u(\D_t,k)$ is unitary. According to Lemma \ref{highderproduct} and Proposition \ref{highderspeh}, the highest shifted derivative of this product is $u(\D_1,k-1)\times \dots \times u(\D_t,k-1)$. It is $\sigma$-distinguished 
according to Proposition \ref{higherdistunitary}, hence, by induction, the product $\D_1\times \dots \times \D_t$ is $\sigma$-distinguished as well.
\end{proof}

In particular, if $u(\D,k)$ is $\sigma$-distinguished, then $\D$ is $\sigma$-distinguished. We are 
now able to prove the main result of this section.

\begin{cor}\label{spehdist}
Let $k$ and $m$ be two positive integers, and $\D$ be a discrete series of $G_m$. The representation $u(\D,k)$ is $\sigma$-distinguished if and only if 
$\D$ is $\sigma$-distinguished.
\end{cor}
\begin{proof}
If $u(\D,k)$ is $\sigma$-distinguished, we already noticed that $\D$ is $\sigma$-distinguished as a consequence of Corollary \ref{uniformproddist}. 
For the converse, we do an induction on $k$.
\par The case $k=1$ is clear, so let's suppose that $u(\D,l)$ is $\sigma$-distinguished for $l\leq k$, 
with $k\geq 1$.
We recall from Theorem \ref{end}, that $\nu^{1/2}u(\D,k)\times \nu^{-1/2}u(\D,k)$ is of length two, and has 
$u(\D_-,k)\times u(\D_+,k)$, and $u(\D,k-1)\times u(\D,k+1)$ as irreducible subquotients. Now, as $u(\D,k)^\vee=u(\D,k)^\sigma$, according 
to the main theorem of \cite{BD08}, 
 the representation $\nu^{1/2}u(\D,k)\times \nu^{-1/2}u(\D,k)$ is $\sigma$-distinguished. But $u(\D_-,k)\times u(\D_+,k)$ can't be distinguished, 
otherwise $\D_-\times \D_+$ would be distinguished thanks to Corollary \ref{uniformproddist}, and this would in turn imply that
both $\D_-$ and $\D_+$ according to Theorem 
 \ref{distgen}, which contradicts Corollary \ref{1sur2}. 
Hence, the representation $u(\D,k-1)\times u(\D,k+1)$ must be $\sigma$-distinguished. We recall that the representation $u(\D,k-1)$ is 
$\sigma$-distinguished 
by induction hypothesis. As $u(\D,k)$ is $\sigma$-distinguished by hypothesis as well, the representation $u(\D,k+1)$ is 
$P_{(k+1)m}^\sigma$-distinguished by Proposition \ref{GtoPspeh}. But then, the representation $u(\D,k+1)$ is $\sigma$-distinguished 
according to Proposition \ref{crux}, and this provides the induction step. 
\end{proof}

As a corollary, we obtain the following result.

\begin{cor}
Let $k$ and $m$ be positive integers. If $\D$ is a segment of $G_m$, and $u(\D,k)^\vee$ is isomorphic to 
$u(\D,k)^\sigma$, then $u(\D,k)$ is either $\sigma$-distinguished, or $(\sigma,\eta)$-distinguished, and not both at the same time.
\end{cor}
\begin{proof}
The representation $u(\D,k)^\vee$ is isomorphic to 
$u(\D,k)^\sigma$ if and only if $\D^\vee$ is isomorphic to $\D^\sigma$. The result is then a consequence of Theorem 7 of \cite{K04} and 
Corollary 1.6. of \cite{AKT04}.
\end{proof}

\subsection{The general case}\label{sectiongeneral}

First, we notice that the class of $\sigma$-induced unitary irreducible representations of $G_n$ is contained in the class of 
$\sigma$-distinguished representations.

\begin{prop}\label{sigmaisdist}
For $n \geq 1$, let $\pi$ be an irreducible unitary representation of $G_n$ which is $\sigma$-induced, then it is $\sigma$-distinguished. 
\end{prop}
\begin{proof}
Let $\D$ be a discrete series of $G_m$ with $m\geq 1$, let $k$ be a positive integer, and let $\alpha$ be a real number. Then the representations 
$\nu_K^\alpha u(\D,k) \times ((\nu_K^\alpha u(\D,k))^\vee)^\sigma$ and $\pi(u(\D,k),\alpha)\times (\pi(u(\D,k),\alpha)^\vee)^\sigma$ are $\sigma$-distinguished according to the main theorem of \cite{BD08}. But as a product 
of $\sigma$-distinguished representations is $\sigma$-distinguished according to Proposition 26 of \cite{F92}, it now follows 
from Remark \ref{onlyrem} that if $\pi$ is $\sigma$-induced, then it is indeed $\sigma$-distinguished.
\end{proof}

It remains to prove the converse, to obtain the main result of this paper. First, we make the following obvious but useful observation.

\begin{LM}\label{dernier}
Let $\pi=u(\D_1,k_1)\times \dots \times u(\D_r,k_r)\times \pi(u(\D_{r+1},k_{r+1}),\alpha_{r+1}) \times \dots \times \pi(u(\D_t,k_t),\alpha_t)$ be an irreducible unitary representation of $G_n$, with $\D_i$ discrete series, and real numbers $\alpha_i$ in $(0,1/2)$. 
If the integers $k_i$ satisfy $k_i\geq 2$, then $\pi$ is $\sigma$-induced if and only if its highest shifted derivative $\pi^{[-]}$ is $\sigma$-induced.
\end{LM}
\begin{proof}
With the notations of the statement, according to to Lemma \ref{highderproduct} and Proposition 
\ref{highderspeh}, the representation $\pi^{[-]}$ is equal to the product $$u(\D_1,k_1-1)\times \dots \times u(\D_r,k_r-1)\times \pi(u(\D_{r+1},k_{r+1}-1),\alpha_{r+1}) \times \dots \times \pi(u(\D_t,k_t-1),\alpha_t).$$ 
Now it is clear that $\pi$ is $\sigma$-self-dual if and only if $\pi^{[-]}$ is, and that a representation $u(\D,k)$ (with $\D$ unitary) occurs with 
odd multiplicity in $\pi$, if and only if $u(\D,k-1)$ occurs with odd multiplicity in $\pi^{[-]}$. 
The result now follows from the fact that a Speh representation $u(\D,k)$ with $k\geq 2$ is $\sigma$-distinguished if and only if $u(\D,k-1)$ is $\sigma$-distinguished thanks to Corollary \ref{spehdist}.
\end{proof}

\begin{thm}\label{unitdist}
If $\pi$ is an irreducible unitary representation of $G_n$, for $n\geq 1$, then $\pi$ is $\sigma$-distinguished if and only it is $\sigma$-induced.
\end{thm}
\begin{proof}
One direction is Proposition \ref{sigmaisdist}. Hence, it remains to show that when $\pi$ is $\sigma$-distinguished, it is $\sigma$-induced. 
To do this, we first write $\pi$ under the form $\pi_1\times \pi_2$ where $\pi_1$ is an irreducible unitary representation of $G_{n_1}$ for some $n_1\geq 0$, which is a product of the form described in the statement of 
Lemma \ref{dernier} (i.e. the $k_i$'s are $\geq 2$), and $\pi_2$ is generic unitary of $G_{n_2}$ for $n_2\geq 0$ (i.e. if you write it as a standard product in the Tadic's classification, all the $k_i$'s are equal to $1$). Notice that $\pi_1$ and $\pi_2$, hence $n_1$ and $n_2$ are uniquely determined by $\pi$. We now prove the statement by induction on $n_1$.\\
\indent The case $n_1=0$ is true thanks to Theorem \ref{distgen}. We thus suppose that $n_1$ is positive, in which case it is necessarily $\geq 2$ by definition of the 
representation $\pi_1$ (the integers $k_i$ occuring in its definition being $\geq 2$), and we suppose that the statement to prove is true for any irreducible unitary 
representation $\pi'=\pi_1'\times \pi_2'$, with $n_1'<n_1$. By hypothesis, the representation $\pi$ is $\sigma$-distinguished, hence the representation  
$\pi^{[-]}=\pi_1^{[-]}$ is $\sigma$-distinguished as well thanks to Proposition \ref{higherdistunitary}. Then, by induction hypothesis, the representation $\pi_1^{[-]}$ must be $\sigma$-induced (because if one writes $\pi'=\pi_1^{[-]}$ under the form $\pi_1'\times \pi_2'$, then we have $n_1'<n_1$). This implies that the representation $\pi_1$ is $\sigma$-induced as well according to Lemma \ref{dernier}, in particular it is $\sigma$-distinguished by Proposition \ref{sigmaisdist}. 
Then, we notice that the representation $\pi_2$ is $P_{n_2}^\sigma$-distinguished according to Proposition  \ref{higherdistunitary}, as 
$\pi_2^{[-]}$ is the trivial character of $G_0$. We can now apply Proposition \ref{crux}, and conclude that $\pi_2$ 
is $\sigma$-distinguished, thus $\sigma$-induced thanks to Theorem \ref{distgen}. This finally implies that 
$\pi$ is $\sigma$-induced as well.
\end{proof}

\end{document}